\documentclass[11pt, reqno, final]{amsart}

\usepackage[utf8]{inputenc}
\usepackage[T1]{fontenc}
\usepackage[english]{babel}
\usepackage{amsmath, amsthm}
\usepackage{ae}
\usepackage{icomma}
\usepackage{units}
\usepackage{color}
\usepackage{graphicx}
\usepackage{bbm}
\usepackage{caption}
\usepackage{array}
\usepackage[hmarginratio=1:1]{geometry}
\usepackage[hyphens]{url}
\usepackage[pdfpagelabels=false, hidelinks]{hyperref}
\usepackage{mathrsfs}
\usepackage{amssymb}
\usepackage{placeins} 
\usepackage{tikz}
\usetikzlibrary{patterns}
\usepackage{float} 
\usepackage[nodayofweek]{datetime}
\usepackage{enumitem}
\usepackage{mathtools}
\usepackage[numbers]{natbib}
\usepackage{dsfont}
\usepackage{ifthen}
\usepackage{comment}

\newcommand{\R}{\ensuremath{\mathbb{R}}}
\newcommand{\N}{\ensuremath{\mathbb{N}}}

\newcommand{\Z}{\ensuremath{\mathbb{Z}}}
\DeclareMathOperator{\dist}{\textnormal{dist}}

\let\eps\varepsilon
\DeclareMathOperator{\sgn}{sgn}

\newtheorem{theorem}{Theorem}[section]

\newtheorem{lemma}[theorem]{Lemma}
\newtheorem{proposition}[theorem]{Proposition}

\numberwithin{theorem}{section}
\numberwithin{definition}{section}
\theoremstyle{remark}
\newtheorem{remark}[theorem]{Remark}

\newcommand{\limplus}{{\mathchoice{\vcenter{\hbox{$\scriptstyle +$}}}
  {\vcenter{\hbox{$\scriptstyle +$}}}
  {\vcenter{\hbox{$\scriptscriptstyle +$}}}
  {\vcenter{\hbox{$\scriptscriptstyle +$}}}
}}

\begin{document}

\title[Discrete Schr\"odinger operators]{Discrete Schr\"odinger operators with decaying and oscillating potentials}

\author[R. L. Frank]{Rupert L. Frank}
\address{\textnormal{(R. L. Frank)} Mathematisches Institut, Ludwig-Maximilians Universit\"at M\"unchen, Theresienstr. 39, 80333 M\"unchen, Germany, and Munich Center for Quantum Science and Technology (MCQST), Schellingstr. 4, 80799 M\"unchen, Germany, and Department of Mathematics, California Institute of Technology, Pasadena, CA 91125, USA}
\email{r.frank@lmu.de, rlfrank@caltech.edu}

\author[S. Larson]{Simon Larson}
\address{\textnormal{(S. Larson)} Mathematical Sciences, Chalmers University of Technology and the University of Gothenburg, SE-41296 Gothenburg, Sweden}
\email{larsons@chalmers.se}

\dedicatory{In memory of Sergey Naboko}

\subjclass[2010]{}

\thanks{\copyright\, 2022 by the authors. This paper may be
reproduced, in its entirety, for non-commercial purposes.\\
Partial support through U.S. National Science Foundation grants DMS-1363432 and DMS-1954995 (R.L.F.), through the Deutsche For\-schungs\-gemeinschaft (DFG, German Research Foundation) through Germany’s Excellence Strategy EXC - 2111 - 390814868 (R.L.F.), and through Knut and Alice Wallenberg Foundation grant KAW~2018.0281 and KAW~2021.0193 (S.L.) are acknowledged.}

\begin{abstract} 
  We study a family of discrete one-dimensional Schr\"odinger operators with power-like decaying potentials with rapid oscillations. In particular, for the potential $V(n)=\lambda n^{-\alpha}\cos(\pi \omega n^\beta)$, with $1<\beta<2\alpha$, we prove that the spectrum is purely absolutely continuous on the spectrum of the Laplacian. 
\end{abstract}

\maketitle

\section{Introduction}

We are interested in the nature of the spectrum of certain Schr\"odinger operators on $\ell^2(\Z_\limplus)$ with decaying and oscillating potentials. Our main interest is the family of Schr\"o\-din\-ger operators whose potential at site $n$ is given by
\begin{equation}\label{eq: V polynomially pseudo random}
  \lambda\frac{\cos(\pi\omega n^\beta)}{n^\alpha}
\end{equation}
where $\omega \in \R\setminus \{0\}$, $\lambda \in \R$, $\alpha\geq 0$, and $\beta >0$ are fixed. It is known that the structure of the spectrum of these Schr\"odinger operators depends in a non-trivial manner on the parameters. The most studied case is that when $\beta=1, \alpha =0$ corresponding to the almost Mathieu operator. Here we are interested in the significantly easier case when $\alpha$ is positive so that the potential decays at infinity. Since in this case the potential tends to zero it immediately follows that the essential spectrum of the operator is equal to the essential spectrum of the Laplacian. While the computation of the essential spectrum is immediate, understanding its precise spectral nature is far from trivial; specifically, we would like to understand the absence or presence of absolutely continuous, singular continuous, and pure point spectrum depending on the parameters of the problem. 

If the fluctuations of the potential induced by the $\cos(\pi \omega n^\beta)$ are instead given by independent random variables with mean zero, it is known that the essential spectrum is almost surely absolutely continuous if $\alpha >1/2$, dense pure point if $\alpha<1/2$, and when $\alpha=1/2$ there is a transition from pure point to purely singular continuous spectrum with respect to the energy and the coupling constant $\lambda$~\cite{Delyon_etal_JPhysA_1983,Simon_CMP_83,DelyonSimonSouillard_PRL84,delyonSimonSouillard_AHP85,KotaniUshiroya_CMP_1988,KiselevLastSimon_EFGP_1998}. If the deterministic oscillations induced by the the cosine suffice to reproduce this behavior is not clear, but it is natural to conjecture that this should be the case for most values of $\omega, \beta$, at least in the non-critical case $\alpha \neq 1/2$. 

While there have been several contributions towards this problem in the past, a full solution is still missing. It is worth mentioning that the continuum Schr\"odinger operator analogue of this problem is completely understood with the final cases settled by White almost 40 years ago~\cite{White_TAMS1983}. For references on the many contributions leading up to the final result, we refer to White's paper; see also~\cite{Jecko}. Unfortunately, the techniques applied in the continuum case which, to a certain degree, all rely on clever changes of variables and integration by parts, have limited applicability in the discrete case. In fact, as we shall discuss later on, the spectral properties of the continuum and discrete models are significantly different at least for certain values of the parameters.

Our modest contribution in this short note is to show that if $1< \beta <2\alpha$, then the spectrum in $(-2, 2)$ is purely absolutely continuous. Under these assumptions it follows from the results of Christ--Kiselev~\cite{ChristKiselev_JAMS_1998} (or rather its discrete analogue), Deift--Killip~\cite{DeiftKillip_CMP1999}, or Remling~\cite{Remling_CMP1998} that the absolutely continuous spectrum coincides with $[-2, 2]$ (see also~\cite{Kiselev_CMP_96,kiselev_stability_1998}). Therefore, our contribution is to show that the essential spectrum is \emph{purely} absolutely continuous. We emphasize that this is not generally the case for Schr\"odinger operators covered by the Christ--Kiselev, Deift--Killip, and Remling result. Indeed, Naboko~\cite{Naboko_86} and Simon~\cite{Simon_PAMS97} have constructed examples of continuum Schr\"odinger operators with potentials $q$ such that $xq(x)$ tends to infinity arbitrarily slowly for which the absolutely continuous spectrum is $[0, \infty)$ and the point spectrum is dense in $[0, \infty)$.

While our main interest is towards the Schr\"odinger operators with potential given by~\eqref{eq: V polynomially pseudo random}, our proof only depends on rather weak properties of the potential and will cover a larger class of Schr\"odinger operators. Specifically we shall consider potentials that for some finite $m$ can be written as
\begin{equation}\label{eq: V representation}
  V(n) = \sum_{j=1}^m \frac{\lambda_j\cos(\phi_j(n))}{n^{\alpha_j}} + V_0(n)
\end{equation}
where 
\begin{enumerate}
  \item $\alpha_j\in (1/2, 1]$ with $\alpha_j\leq \alpha_{j+1}$,
  \item $\phi_j$ can be written as $\phi_j = \phi_j^0 + \phi_j^1$ where $\phi_j^0 \in C^2(\R_\limplus)$ and there exist $\beta_j \in (1, 2\alpha_j)$, $\omega_j \in \R\setminus \{0\}$, and $\gamma_j >1-\alpha_j$ such that
  \begin{align}\label{eq: phase assumption}
    |\partial^k_x(\phi_j^0(x) - \pi \omega_j x^{\beta_j})|&= o_{x\to \infty}(x^{\beta_j-k})\quad \mbox{for }k=0, 1, 2, \, \mbox{and}\\
    |\phi_j^1(n+1)-\phi_j^1(n)| &\lesssim n^{-\gamma_j}\,, \nonumber
  \end{align}
  
  \item $\lambda_j \in \R$, and
  \item $\{V_0(n)\}_{n\geq 1}\in \ell^1$.
\end{enumerate}
We emphasize that the important assumption here is that $1<\beta_j<2\alpha_j$, so that $V$ is a linear combination of potentials asymptotically resembling~\eqref{eq: V polynomially pseudo random} with $1<\beta< 2\alpha$. The assumption that $\alpha_j\leq 1$ is only for later convenience. Indeed, if $\alpha_j>1$ for some $j$ then this term can be included in $V_0$.

For a boundary condition $\mu \in \R$ and a potential $V$ define the half-line Schr\"odinger operator $H_{V}^{\mu} = \Delta + V$ in $\ell^2(\Z_\limplus)$ by
\begin{equation*}
  (H_{V}^{\mu}\psi)_n=
  \begin{cases}
    \psi_{n+1}+\psi_{n-1} + V(n)\psi_n\,, & n\geq 2\\
    \psi_0 + (V(1)+\mu)\psi_1\,, & n=1\,.
  \end{cases}
\end{equation*}
Since $\lim_{n\to \infty}V(n)=0$ it holds that $\sigma_{ess}(H^\mu_V)=\sigma_{ess}(\Delta)=[-2, 2]$.

Our main result is the following theorem.
\begin{theorem}\label{thm: pure ac}
  If $V$ is of the form~\eqref{eq: V representation}, then $H_V^\mu$ has purely absolutely continuous spectrum in $(-2, 2)$.
\end{theorem}

To clarify how our contribution fits into existing results for $V$ given by~\eqref{eq: V polynomially pseudo random}, we gather what we have been able to find in the literature in the following proposition.
\begin{proposition}\label{prop: known results}
  Let $H_{\omega, \alpha, \beta}^{\lambda, \mu}$ be the Schr\"odinger operator with potential
  \begin{equation*}
    V(n) = \lambda \frac{\cos(\pi \omega n^\beta)}{n^\alpha}  
  \end{equation*}
  and boundary condition parametrized by $\mu \in \R$.
  It holds that
  \begin{enumerate}
    \item If $\beta \in (0, 1)$ and $\alpha >0$, then $H^{\lambda,\mu}_{\omega,\alpha,\beta}$ has purely absolutely continuous spectrum in $(-2, 2)$.

    \item If $\beta=1$ and $\alpha >0$, then $\sigma_{ac}(H^{\lambda,\mu}_{\omega,\alpha,1})= [-2, 2]$, $\sigma_{sc}(H^{\lambda,\mu}_{\omega,\alpha,1})= \emptyset$, and 
    $$
    \sigma_{pp}(H^{\lambda,\mu}_{\omega,\alpha,1})\cap [-2, 2] \subseteq \{2\cos(j\omega/2): j =0, \ldots, \lfloor 1/\alpha\rfloor\}\,.
    $$

    \item If $\beta\in (1, 2)$ and $0< \alpha < \frac{2-\beta}{2}$, then $\sigma_{ac}(H^{\lambda,\mu}_{\omega,\alpha,\beta})=\emptyset$. Furthermore, for almost every (but not every) $\mu$ the spectrum is pure point.

    \item If $\alpha >1/2$, then $ \sigma_{ac}(H^{\lambda,\mu}_{\omega,\alpha,\beta}) = [-2, 2]$.

    \item If $\alpha >1$, then $H^{\lambda,\mu}_{\omega,\alpha,\beta}$ has purely absolutely continuous spectrum in $(-2, 2)$.
  \end{enumerate}
\end{proposition}

The first statement in Proposition~\ref{prop: known results} is due to Stolz~\cite{Stolz_SlowlyOsc_1994} whose result applies to a much larger class of slowly oscillating potentials. The second part of the proposition, i.e.\ the case $\beta=1$, is due to Lukic~\cite{Lukic_CMP_2011}. Again the results of Lukic extend to a more general form of potentials.  

If $\alpha>1/2$, then as noted earlier Proposition~\ref{prop: known results} (4) follows from (the discrete analogue of) the results in~\cite{ChristKiselev_JAMS_1998,Remling_CMP1998,DeiftKillip_CMP1999}. We note again that this is also the case for potentials with the representation in~\eqref{eq: V representation} and our contribution to the statement of Theorem~\ref{thm: pure ac} is that when $1<\beta<2\alpha$ the spectrum in $(-2, 2)$ is \emph{purely} absolutely continuous.  

The fact that the essential spectrum is purely absolutely continuous when $\alpha >1$, so that $V\in \ell^1$, follows from general results that only depend on the summability of the potential.

Finally, Proposition~\ref{prop: known results} (3) is due to Kr\"uger~\cite{Kruger_2011}. While Kr\"uger only considers the case $\omega=2, \lambda=1$ his proof goes through in the general case with obvious modifications. Since we have not been able to verify the proof of a necessary lemma in~\cite{Kruger_2011}, we include an appropriately modified proof in Appendix~\ref{app: correction to Kruger}. Since the absolutely continuous component of the spectrum is independent of the boundary condition Kr\"uger's result implies that $\sigma_{ac}(H^{\lambda, \mu}_{\omega, \alpha, \beta})=\emptyset$ for all $\mu$. However, the theory of rank-one perturbations implies that for a dense $G_\delta$ set of $\mu$ the singular continuous component of the spectrum is non-empty. In particular, as claimed in the proposition, almost every $\mu$ cannot be improved to every $\mu$. 

In stark contrast to Kr\"uger's result for the discrete model, for the continuum model with $\alpha, \beta$ as in Proposition~\ref{prop: known results} (3) White~\cite{White_TAMS1983} proves that the spectrum is purely absolutely continuous on the spectrum of the Laplacian. 

We note that, for $\alpha=1$ and $\beta=1$, the potential in Proposition 1.2 is a discrete variant of the famous Wigner-von Neumann potential, which provides an example of a Schr\"odinger operator with an embedded eigenvalue. This has led to an enormous literature, and we restrict ourselves to mentioning only \cite{vNeumannWigner_1929,ReedSimon_IV,EasthamKalf,FrankSimon_17}.

The proof of Theorem~\ref{thm: pure ac} is based on Gilbert--Pearson's subordinacy theory. Effectively, this reduces the problem to understanding the asymptotic behavior at infinity of solutions of the eigenvalue equation for our operator. The first step to obtain the desired asymptotics is to rewrite the eigenvalue equation in suitable Pr\"ufer-type coordinates; this is the content of Section~\ref{sec: Prufer}. A key part of the analysis of the Pr\"ufer equations consists in the estimating certain exponential sums; this is the topic of Section~\ref{sec: Expsum bounds}. 

\medskip

{\noindent \bf Acknowledgements} The first author is grateful to J.\ Breuer and H.\ Kr\"uger for discussions on the topic of this paper.

\section{Pr\"ufer coordinates}
\label{sec: Prufer}

Let $\{\psi_n\}_{n\geq 1}$ be the solution of the eigenvalue equation at energy $E =2\cos(k)$ with $0<k<\pi$, that is,
\begin{equation*}
    \begin{cases}
        \psi_{n+1}+\psi_{n-1}+\lambda V_{\omega,\alpha,\beta}(n)\psi_n = 2\cos(k)\psi_n & \mbox{for }n\geq 2\,,\\
       \psi_2 + (\lambda V_{\omega,\alpha,\beta}(1)+\mu)\psi_1 = 2\cos(k)\psi_1\,.
    \end{cases}
\end{equation*}
Setting $\psi_0 = \mu\psi_1$ the second equation can be absorbed into the first by extending it to $n=1$. 

For $n \geq 1$ define the (modified) Pr\"ufer variables $R\in \R_\limplus, \eta\in \R$ by
\begin{align*}
  R(n)\cos(\eta(n)+k(n-1))&=\psi_n-\cos(k)\psi_{n-1}\,,\\
  R(n)\sin(\eta(n)+k(n-1))&=\sin(k)\psi_{n-1}\,.
\end{align*}
The ambiguity in $\eta$ is fixed by demanding $\eta(1), \eta(n+1)-\eta(n) \in [-\pi, \pi)$. For notational convenience we also set $\theta(n)= \eta(n)+kn$.

We note that the boundary condition $\psi_0 = \mu \psi_1$ becomes
\begin{align*}
  R(1) &= |\psi_1|(1+\mu^2-2\mu \cos(k))^{1/2}\,,\\
  \eta(1) &= \arccos\Bigl(\frac{\sgn(\psi_1)(1-\mu \cos(k))}{(1+\mu^2-2\mu \cos(k))^{1/2}}\Bigr)\,.
\end{align*}

Setting $\nu_k(n)= -\frac{V(n)}{\sin(k)}$, the variables $R, \eta$ satisfy
\begin{equation}\label{eq: Discrete Prufer}
\begin{aligned}
  R(n+1)^2 &= R(n)^2\Bigl[1 + \nu_k(n)\sin(2\theta(n))+  \nu_k(n)^2\sin^2(\theta(n))\Bigr] \,,\\
  \cot(\eta(n+1)+kn) &= \cot(\eta(n)+kn)+\nu_k(n)\,,
\end{aligned}
\end{equation}
see for instance~\cite{KiselevLastSimon_EFGP_1998} (note that our $\theta$ corresponds to their $\tilde \theta$).

Under our assumptions $\nu_k(n)$ decays in $n$ and it becomes more practical to deal with approximate equations obtained by appropriate Taylor expansions. Specifically, we shall base our arguments on the asymptotic equations
\begin{equation}\label{eq: approx Prufer eqs}
\begin{aligned}
  \log\Bigl(\frac{R(n+1)}{R(n)}\Bigr) &=
  \frac{\nu_k(n)}{2}\sin(2\theta(n)) + 
  O(|\nu_k(n)|^2)\,,\\
  \eta(n+1) - \eta(n)
  &=
   -\frac{\nu_k(n)}{2}+\frac{\nu_k(n)}{2}\cos(2\theta(n))+O(|\nu_k(n)|^2)\,.
\end{aligned}
\end{equation}
This form of the equations is well suited to study how the solutions behave on scales larger then $1$.
Naturally, one could have made Taylor expansions to higher order but for our purposes the $|\nu_k(n)|^2$ error term will be sufficient. Under our assumptions this error is summable and the approximate equation for $R$ determines its asymptotic behavior up to a bounded multiplicative factor, which is sufficient for the application of Gilbert--Pearson subordinacy theory~\cite{GilbertPearson_subordinacy_1987}. While it is rather clear that the above equations provide sufficient information concerning $R$, it need not be the case that this precision suffices when it comes to $\eta$. Indeed, from the manner in which $\eta$ enters the equation for $R$ even a bounded additive perturbation might without further information have a very large effect. As we shall see, this is not the case for us. Indeed, we will argue below that the oscillations of $\sin(2\theta(n))$ can be treated as insignificant in relation to those induced by $\nu_k(n)$. However, if one were to attempt to extend our techniques to the region where $\alpha_j >1/2$ and $\beta_j \geq 2\alpha_j$, then one might need to study the equation for $\eta$ to higher precision.

By iterating~\eqref{eq: approx Prufer eqs} and inserting the assumed form of $V$, one finds that for integers $1\leq N_1< N_2$
\begin{equation}\label{eq: R eq iterated}
  \log\Bigl(\frac{R(N_2+1)}{R(N_1)}\Bigr) 
  =
  -\sum_{j=1}^m\sum_{n=N_1}^{N_2} \!\!\frac{\lambda_j\cos(\phi_j(n))\sin(2\theta(n))}{2\sin(k)n^{\alpha_j}} + O\Biggl(\sum_{n=N_1}^{N_2}\!\Bigl[\frac{1}{n^{2\alpha_1}} + |V_0(n)|\Bigr]\!\Biggr)\,,
\end{equation}
and
\begin{equation}\label{eq: eta eq iterated}
  \eta(N_2+1) - \eta(N_1)
  =
   \sum_{j=1}^m\sum_{n=N_1}^{N_2}\!\! \frac{\lambda_j\cos(\phi_j(n))}{2\sin(k)n^{\alpha_j}}\Bigl[1-\cos(2\theta(n))\Bigr] +  O\Biggl(\sum_{n=N_1}^{N_2}\!\Bigl[\frac{1}{n^{2\alpha_1}} + |V_0(n)|\Bigr]\!\Biggr)\,.
\end{equation}

Our goal is to show that these equations imply that the Pr\"ufer radius $R(n)$ has a finite and non-zero limit as $n \to \infty$. By Gilbert--Pearson subordinacy theory, it follows that the spectrum of $H_{V}^{\mu}$ in $(-2, 2)$ is purely absolutely continuous. Given the form of the equations above it should not come as a surprise that the key to our analysis is to understand certain exponential sums. The necessary exponential sum bounds are the topic of the next section. 

Under our assumptions we shall see that the only information that we need concerning $\eta$ is the rough continuity estimate
\begin{equation}\label{eq: eta continuity bound}
  |\eta(n+1)-\eta(n)|\lesssim n^{-\alpha_1} + |V_0(n)|\,.
\end{equation}
If one were to attempt to use the same ideas as employed here under weaker assumptions on the pairs $(\alpha_j, \beta_j)$ one might very well need to consider the equations for $R$ and $\eta$ to higher precision. For a problem closely resembling the case $\alpha=1/2, \beta=3/2$ such analysis was recently carried out by the authors in~\cite{FrankLarson_KP_2021}. It is likely that the methods used in that paper could successfully be applied also to the current setting to treat a larger region of $(\alpha, \beta)$. However, without further insights this analysis would likely require certain number theoretic assumptions on the parameter $\omega$ which we believe are artificial when $\alpha>1/2$. However, such number theoretic assumptions might very well be relevant when $\alpha\leq 1/2$.

\section{Exponential sum bounds}
\label{sec: Expsum bounds}

As we observed in the previous section our analysis will rely on understanding exponential sums with power decaying amplitudes and phases given by combinations of $\phi_j$ and~$\theta$. The goal of this section is to prove that we can understand these sums to fairly high accuracy by treating the dependence on the Pr\"ufer angle $\eta$ perturbatively.

More specifically we shall consider sums of the form
\begin{equation*}
  \sum_{a<n\leq b} \frac{e^{i(\phi(n) + h(n))}}{n^\rho}
\end{equation*}
where $\phi \approx \pi \omega n^\beta$ for some $\omega, \beta$ (in the sense of~\eqref{eq: phase assumption}) and $h$ satisfies a regularity estimate in terms of the differences
\begin{equation}\label{eq: h differences}
  \delta h(n) =|h(n+1)-h(n)|\,.
\end{equation}

As we shall see there are two natural scales appearing for these exponential sums. One is that determined by the known oscillations $\phi(n)= \pi \omega n^{\beta}(1+o(1))$ corresponding to considering $n$ comparable to $l^{1/(\beta-1)}$ for some integer $l$. This corresponds to the distribution of points $x\in \R$ where $\phi'(x)  \in 2\pi \Z$ (where the local oscillation rate lines up with the integer lattice). The second scale comes from the regularity of the perturbation $h$. In our case the perturbation will essentially satisfy an estimate of the form $\delta h(n) \lesssim n^{-\gamma}$ for some $\gamma >0$ (cf.~\eqref{eq: eta continuity bound}). If $2-\beta <\gamma$, then $h$ can essentially be treated as a constant on the scale determined by the oscillations. If $2-\beta>\gamma$ then this might no longer be the case, which will lead to some complications that need to be dealt with.

The main theorem to be proven in this section is the following.
\begin{theorem}\label{thm: main exp sum bound}
 Fix $1<\beta <2,$ $\rho \geq 0$, $\gamma \in [0, 1]$, and $\eps>0$. Let $\phi\in C^2(\R_\limplus)$ be a real-valued function satisfying
  \begin{equation*}
    |\partial^k(\phi(x)-\pi \omega x^\beta)| \lesssim o_{x\to \infty}(|x|^{\beta-k}) \mbox{ for } k=0, 1, 2 \mbox{ and some } \omega\neq 0\,,
  \end{equation*}
  and $h\colon \N \to \R$ satisfy
  \begin{equation*}
    |h(n+1)-h(n)| \lesssim n^{-\gamma} + v_n \quad \mbox{with } v=\{v_n\}_n \in \ell^1(\N)\,.
  \end{equation*}
  For $l\in \N$ large let $Y_l$ be the unique solution of $\phi'(Y_l) = \sgn(\omega)2\pi l$. Then
  \begin{equation*}
    \biggl|\sum_{Y_l<n\leq Y_{l+1}} \frac{e^{i(\phi(n)+h(n))}}{n^\rho}\biggr| \lesssim \Bigl(1+\|v\|_{\ell^1} + l^{\frac{2-\beta-2\gamma}{2(\beta-1)}+\eps}\Bigr)l^{\frac{2-\beta-2\rho}{2(\beta-1)}}\,.
  \end{equation*}
\end{theorem}

\begin{remark} 
A few remarks:
\begin{enumerate}
  \item We emphasize that the regions close to the points $Y_l$ in Theorem~\ref{thm: main exp sum bound} have a special significance for the size of the exponential sum and as such the appearance of these points in the bounds is not coincidental.

  \item By application of the Poisson summation formula and a stationary phase calculation it is not very difficult to verify that the contribution to the sum from an $l^{\frac{2-\beta}{2(\beta-1)}}$ neighbourhood of $Y_l$ is comparable to $l^{\frac{2-\beta-2\rho}{2(\beta-1)}}$. In particular, the order of the first term in the bound is sharp. If the dependence of the bound on $\gamma$ is best possible is less clear. 

  \item Noting that $Y_l = \bigl(\frac{2l}{\beta|\omega|}\bigr)^{\frac{1}{\beta-1}}(1+o(1))$ and $Y_{l+1}-Y_l = \frac{1}{\beta-1}\bigl(\frac{2}{\beta|\omega|}\bigr)^{\frac{1}{\beta-1}}l^{\frac{2-\beta}{\beta-1}}(1+o(1))$ the sum is trivially $O(l^{\frac{2-\beta-\rho}{\beta-1}})$. Similarly, a square-root cancellation heuristic (treating the phases $e^{i(\phi(x)+h(n))}$ as i.i.d.\ random variables with mean zero) suggests that we cannot expect the sum to be much smaller than $O(l^{\frac{2-\beta-2\rho}{2(\beta-1)}})$. This again shows the optimality of the first term in the bound.
\end{enumerate}
  
\end{remark}

The above theorem will be obtained through fairly standard arguments; away form the $Y_l$ we shall employ the classical bound of Kuzmin--Landau, close to the $Y_l$ we bound the terms trivially; when combined with a simple optimization over choices of scales this leads to the above theorem. Essentially, we are redoing the standard argument to obtain the classical bound of van der Corput (see for instance~\cite[Theorem~5.9]{Titchmarsh_RiemanZeta86}) but allowing for the additional amplitude $n^{-\rho}$ and the additional perturbation of the phase by $h$.

While our main aim in this section is to prove Theorem~\ref{thm: main exp sum bound} it is convenient to start by proving much simpler bounds. Specifically, we shall work our way up from direct consequences of the classical bound of Kuzmin--Landau by step-by-step introducing the complications of the decaying amplitude and the unknown phase.

\begin{lemma}\label{lem: no decay exp estimate}
  Fix $1<\beta <2$. Let $\phi, Y_l$ be as in Theorem~\ref{thm: main exp sum bound}. Then, for $Y_l \leq a<b \leq Y_{l+1}$,
  \begin{equation*}
    \biggl|\sum_{a <n \leq b} e^{i\phi(n)}\biggr| \lesssim 
    l^{\frac{2-\beta}{\beta-1}}\dist(\{Y_l, Y_{l+1}\}, [a, b])^{-1}\,.
  \end{equation*}
\end{lemma}

\begin{proof}
    Recall the classical result of Kuzmin--Landau~\cite[Lemma~4.19]{Titchmarsh_RiemanZeta86}: If $\min_{\nu\in \Z}|f'(x)-\nu|\geq \kappa >0$ for all $x \in [a, b]$ then
    \begin{equation}\label{eq: Kuzmin-Landau}
      \biggl|\sum_{a<n\leq b}e^{2\pi i f(n)}\biggr| \lesssim \kappa^{-1}\,.
    \end{equation}
    By asymptotic monotonicity of $\phi'$,
    $$
    \min_{\nu \in \Z} \Bigl|\frac{\phi'(n)}{2\pi}-\nu\Bigr| = \min\Bigl\{\Bigl|\frac{\phi'(n)}{2\pi \sgn(\omega)}-l\Bigr|, \Bigl|\frac{\phi'(n)}{2\pi\sgn(\omega)}-l-1\Bigr|\Bigr\}\quad \mbox{for all } x \in [Y_l, Y_{l+1}]\,.
    $$
    For $n\geq Y_l + cl^\sigma$, Taylor expanding around $Y_l$
    \begin{equation*}
      \phi'(n) = \phi'(Y_l) + \phi''(z)(n-Y_l)= \sgn(\omega)2\pi l + \phi''(z)(n-Y_l) \quad \mbox{for some }z \in [Y_l, n]\,.
    \end{equation*}
    Since $n-Y_l \geq \dist(\{Y_l, Y_{l+1}\}, [a, b])$ and $\phi''(z) \gtrsim l^{\frac{\beta -2}{\beta-1}}$ we deduce that for $n$ in the range of the sum 
    \begin{equation*}
      \Bigl|\frac{\phi'(n)}{2\pi\sgn(\omega)}-l\Bigr| \gtrsim l^{\frac{\beta-2}{\beta-1}}\dist(\{Y_l, Y_{l+1}\}, [a, b])\,.
    \end{equation*}
    The analogous argument proves that also
    \begin{equation*}
      \Bigl|\frac{\phi'(n)}{2\pi\sgn(\omega)}-l-1\Bigr| \gtrsim l^{\frac{\beta-2}{\beta-1}}\dist(\{Y_l, Y_{l+1}\}, [a, b])\,.
    \end{equation*}
    The bound claimed in the lemma follows from~\eqref{eq: Kuzmin-Landau}.
\end{proof}

Let's add decay and perturb the phase.
\begin{lemma}\label{lem: decay exp sum1}
  Fix $1<\beta <2$ and $\rho \geq 0$. Let $\phi, Y_l$ be as in Theorem~\ref{thm: main exp sum bound} and define $\delta h$ by~\eqref{eq: h differences}. If $Y_l \leq a<b \leq Y_{l+1}$, then
  \begin{equation*}
    \biggl|\sum_{a<n\leq b} \frac{e^{i(\phi(n)+h(n))}}{n^\rho}\biggr| 
    \lesssim \Bigl(1+\|\delta h\|_{\ell^1(a, b)}\Bigr)
    l^{\frac{2-\beta-\rho}{\beta-1}}\dist(\{Y_l, Y_{l+1}\}, [a, b])^{-1}\,.
  \end{equation*}
\end{lemma}

\begin{proof}
   Without loss of generality assume that $a, b \in \N$. Summing by parts
   \begin{equation}\label{eq: sum by parts}
   \begin{aligned}
      \sum_{a <n \leq b} \frac{e^{i(\phi(n)+h(n))}}{n^\rho} 
      &=
      \sum_{a <n \leq b-1}\biggl[\frac{e^{ih(n)}}{n^\rho}-\frac{e^{ih(n+1)}}{(n+1)^\rho}\biggr] \sum_{a<j\leq n}e^{i\phi(j)}\\
      &\quad 
      + \frac{e^{i h(b)}}{b^\rho}\sum_{a <n \leq b} e^{i\phi(n)}\,.
    \end{aligned}
    \end{equation} 
    Using the assumption on $h$,
    \begin{align*}
      \biggl|\frac{e^{ih(n)}}{n^\rho}-\frac{e^{ih(n+1)}}{(n+1)^\rho}\biggr|
      &\leq
      \biggl|\frac{1}{n^\rho}-\frac{1}{(n+1)^\rho}\biggr|+ \frac{|h(n)-h(n+1)|}{n^\rho}\\
      &\lesssim n^{-\rho-1} + n^{-\rho}\delta h(n)\,.
    \end{align*}
    Inserting this bound into~\eqref{eq: sum by parts} we arrive at
    \begin{align*}
      \biggl|\sum_{a <n \leq b} \frac{e^{i(\mu p(n)+h(n))}}{n^\rho} \biggr|
      &\lesssim
      \sum_{a <n \leq b-1}n^{-\rho}(n^{-1}+\delta h(n))\biggl|\sum_{a<j\leq n}e^{i\phi(j)}\biggr|\\
      &\quad 
      + b^{-\rho}\biggl|\sum_{a <n \leq b} e^{i\phi(n)}\biggr|\,.
    \end{align*}
     By Lemma~\ref{lem: no decay exp estimate}, H\"older's inequality and the fact that $Y_l \gtrsim l^{\frac{1}{\beta-1}}$,
     \begin{align*}
       \biggl|\sum_{a <n \leq b} \frac{e^{i(\phi(n)+h(n))}}{n^\rho} \biggr|
      &\lesssim
      \biggl(b^{-\rho} +  \sum_{a<n\leq b-1}n^{-\rho}(n^{-1}+\delta h(n))\biggr)
      l^{\frac{2-\beta}{\beta-1}}\dist(\{Y_l, Y_{l+1}\}, [a, b])^{-1}\\
      &\lesssim
      \Bigl(1+\|\delta h\|_{\ell^1(a, b)}\Bigr)
      l^{\frac{2-\beta-\rho}{\beta-1}}\dist(\{Y_l, Y_{l+1}\}, [a, b])^{-1}\,.
     \end{align*}
     Here we used $\sum_{a<n<b}n^{-1} \lesssim \log(b/a) \leq \log(Y_{l+1}/Y_l) \lesssim 1$.
    This completes the proof of the lemma.
\end{proof}

With the above bound in hand we are ready to prove Theorem~\ref{thm: main exp sum bound}.

\begin{proof}[Proof of Theorem~\ref{thm: main exp sum bound}]
  The idea of the proof is to split the summation range $(Y_l, Y_{l+1}]$ into appropriate subintervals and apply either the trivial bound or the bound in Lemma~\ref{lem: decay exp sum1} to the individual pieces. The choice of subintervals will be chosen depending on the distance to the resonant points $Y_l, Y_{l+1}$. To simplify notation we only write out the argument for the sum over the first half of the interval. The sum over the second half can be treated analogously.

  Recall that $Y_{l+1}-Y_{l}$ is proportional to $l^\frac{2-\beta}{\beta-1}$. Our goal is to construct a finite increasing sequence $\{\sigma_k\}_{k=1}^K$ with $0<\sigma_j<\frac{2-\beta}{\beta-1}$ and split our sum over the intervals $I_0 = (Y_l, Y_l+l^{\sigma_k}]$, $I_K = (Y_l+l^{\sigma_K}, \frac{Y_{l+1}-Y_l}{2}]$, and $I_k=(Y_l+l^{\sigma_{k}}, Y_l + l^{\sigma_{k+1}}]$ with $k=1, \ldots, K-1$. That is, we write
  \begin{align*}
    \sum_{Y_l < n \leq \frac{Y_{l+1}-Y_l}{2}} \frac{e^{i(\phi(n)+h(n))}}{n^\rho} &= \sum_{k=0}^{K}\sum_{n \in I_k} \frac{e^{i(\phi(n)+h(n))}}{n^\rho}\,.
  \end{align*}
  For $k=1, \ldots, K-1$, it holds that $|I_k| \lesssim l^{\sigma_{k+1}}$ and $\dist(I_k, Y_l) \gtrsim l^{\sigma_k}$. Similarly, $|I_0|\lesssim l^{\sigma_1}, |I_K| \lesssim l^{\frac{2-\beta}{\beta-1}},$ and $\dist(I_K, Y_l) \gtrsim l^{\sigma_K}$.

  Using the trivial bound for the sum over $I_0$ and Lemma~\ref{lem: decay exp sum1} for the sum over the $I_k$, $k>0$, we have the bound
  \begin{align*}
    \biggl|\sum_{Y_l < n \leq \frac{Y_{l+1}-Y_l}{2}} \frac{e^{i(\phi(n)+h(n))}}{n^\rho}\biggr| 
    &\lesssim 
    \sum_{n\in I_0}\frac{1}{n^\rho}+
    \sum_{k=1}^{K}\Bigl(1+\|\delta h\|_{\ell^1(I_k)}\Bigr)l^{\frac{2-\beta-\rho}{\beta-1}-\sigma_{k}}\,.
  \end{align*}
  Since $\delta h(n) \lesssim n^{-\gamma}+v_n$, $K$ is finite, and $Y_l\gtrsim l^{\frac{1}{\beta-1}}$, it follows that
  \begin{align*}
    \biggl|\sum_{Y_l < n \leq \frac{Y_{l+1}-Y_l}{2}} \frac{e^{i(\phi(n)+h(n))}}{n^\rho}\biggr| 
    &\lesssim l^{\sigma_1 - \frac{\rho}{\beta-1}}
    +
    (1+\|v\|_{l^1})l^{\frac{2-\beta-\rho}{\beta-1}-\sigma_1}\\
    &\quad 
    +l^{\frac{2-\beta-\rho-\gamma}{\beta-1}}\sum_{k=1}^{K-1}l^{\sigma_{k+1}-\sigma_{k}} + l^{\frac{2-\beta-\rho-\gamma}{\beta-1}+ \frac{2-\beta}{\beta-1}-\sigma_K}\,.
  \end{align*}
  Choosing $\sigma_1 = \frac{2-\beta}{2(\beta-1)}$ and setting $\sigma_{k+1}-\sigma_k = \frac{1}{K}(\frac{2-\beta}{\beta-1}-\sigma_1) = \frac{2-\beta}{2K(\beta-1)}$ with the smallest positive integer so that $\frac{2-\beta}{2K(\beta-1)}\leq \eps$, the left-hand side above is bounded by 
  \begin{align*}
    (1+\|v\|_{l^1})l^{\frac{2-\beta-2\rho}{2(\beta-1)}}
    +l^{\frac{2-\beta-\rho-\gamma}{\beta-1} + \eps}
    =
    \Bigl(1+\|v\|_{l^1} + l^{\frac{2-\beta-2\gamma}{2(\beta-1)}+\eps}\Bigr)l^{\frac{2-\beta-2\rho}{2(\beta-1)}}\,.
  \end{align*}
  This completes the proof of Theorem~\ref{thm: main exp sum bound}.
\end{proof}

\section{Proof of Theorem~\ref{thm: pure ac}}

Our proof of Theorem~\ref{thm: pure ac} is based on the subordinacy theory of Gilbert--Pearson~\cite{GilbertPearson_subordinacy_1987}. Specifically, that spectrum of $\Delta+V$ in $(-2, 2)$ is purely absolutely continuous will follow once we show that for every $k\in (0, \pi)$ and $\mu \in \R$ the Pr\"ufer radius $R$ associated to the solution of the generalized eigenvalue equation has a finite and non-zero limit as $n$ tends to infinity.

Fix $k\in (0, \pi), \mu \in \R$ and let $R$ be the Pr\"ufer radius associated to the corresponding solution of the generalized eigenvalue equation. By scaling we may without loss of generality assume that $R(1)=1$. To prove that $R$ has a non-zero and finite limit as $n \to \infty$ we aim to show that
\begin{equation*}
  \lim_{n\to \infty} \log(R(n))
\end{equation*}
exists and is finite.

We argue that $\{\log(R(n))\}_{n\geq 1}$ is a Cauchy sequence. By~\eqref{eq: R eq iterated},
\begin{align*}
  \log(R(N_2+1))&-\log(R(N_1)) \\
  &=
  -\sum_{j=1}^m\sum_{n=N_1}^{N_2} \frac{\lambda_j\cos(\phi_j(n))\sin(2\theta(n))}{2\sin(k)n^{\alpha_j}} + O\Biggl(\,\sum_{n=N_1}^{N_2}\!\Bigl[\frac{1}{n^{2\alpha_1}} + |V_0(n)|\Bigr]\!\Biggr)\,.
\end{align*}
Since by assumption $\{V_0(n)\}_{n\geq 1}\in \ell^1$ and $\alpha_1>1/2$ the quantity in the error term is $o_{N_1 \to \infty}(1)$ and in particular it constitutes a Cauchy sequence.

To see that the main contribution is also Cauchy we argue as follows. Consider a fixed $j$ and write the corresponding sum as
\begin{align*}
  \sum_{n=N_1}^{N_2} &\frac{\cos(\phi_j(n))\sin(2\theta(n))}{n^{\alpha_j}}\\ &= \frac{1}{2}\Im\Biggl[\sum_{n=N_1}^{N_2} \frac{e^{i(\phi_j(n)+2\theta(n))}}{n^{\alpha_j}}\Biggr]-\frac{1}{2}\Im\Biggl[\sum_{n=N_1}^{N_2} \frac{e^{i(\phi_j(n)-2\theta(n))}}{n^{\alpha_j}}\Biggr]\\
  &=
  \frac{1}{2}\Im\Biggl[\sum_{n=N_1}^{N_2} \frac{e^{i(\phi^0_j(n)+2kn+\phi_j^1(n)+2\eta(n))}}{n^{\alpha_j}}\Biggr]-\frac{1}{2}\Im\Biggl[\sum_{n=n_1}^{N_2} \frac{e^{i(\phi^0_j(n)-2kn+\phi_j^1(n)-2\eta(n))}}{n^{\alpha_j}}\Biggr]\,.
\end{align*}
Since the argument is identical for both terms, we consider only the first.

Define $\phi, h$ by $\phi(n) = \phi_j^0(n)+2kn$ and $h(n) = \phi_j^1(n)+2\eta(n)$. By our assumption on $\phi_j^0$, $\phi$ satisfies the assumptions of Theorem~\ref{thm: main exp sum bound} with $\beta=\beta_j$ and $\omega=\omega_j$. Similarly, by the assumptions on $\phi_j^1$ and the bound~\eqref{eq: eta continuity bound} for $\eta$ we conclude that
\begin{equation*}
  |h(n+1)-h(n)| \lesssim n^{-\min\{\gamma_j, \alpha_1\}} + |V_0(n)|
\end{equation*}
so $h$ satisfies the assumptions of Theorem~\ref{thm: main exp sum bound} with $\gamma = \min\{\gamma_j, \alpha_1\}$ and $v_n = |V_0(n)|$. Let $\{Y_l\}_{l \geq 1}$ be defined as in the theorem, set $L_1 = \min\{l: Y_l> N_1\}$ and $L_2 = \max\{l: Y_l< N_2\}$. Then
\begin{align*}
  \sum_{n=N_1}^{N_2} \frac{e^{i(\phi^0_j(n)+2kn+\phi_j^1(n)+2\eta(n))}}{n^{\alpha_j}}
  &=
  \sum_{l=L_1}^{L_2}\sum_{Y_{l-1}< n\leq  Y_l} \frac{e^{i(\phi^0_j(n)+2kn+\phi_j^1(n)+2\eta(n))}}{n^{\alpha_j}}\\
  &\quad
  +\sum_{N_1\leq n \leq Y_{L_1}} \frac{e^{i(\phi^0_j(n)+2kn+\phi_j^1(n)+2\eta(n))}}{n^{\alpha_j}}\\
  &\quad 
  +\sum_{Y_{L_2}< n\leq N_2} \frac{e^{i(\phi^0_j(n)+2kn+\phi_j^1(n)+2\eta(n))}}{n^{\alpha_j}}\,.
\end{align*}
The last two sums are by the trivial bound no larger than $O(L_1^{(2-\beta_j-\alpha_j)/(\beta_j-1)})=o(L_1^{-1/2})$, since $\alpha_j>\beta_j/2$. By applying Theorem~\ref{thm: main exp sum bound} to the remaining sums we conclude that, for any $\eps>0$,
\begin{align*}
  \biggl|\sum_{l=L_1}^{L_2}\sum_{Y_{l-1}< n\leq  Y_l} \frac{e^{i(\phi^0_j(n)+2kn+\phi_j^1(n)+2\eta(n))}}{n^{\alpha_j}}\biggr|
  &\lesssim 
  (1+\|V_0\|_{\ell^1})\sum_{l=L_1}^{L_2}
  l^{\frac{2-\beta_j-2\alpha_j}{2(\beta_j-1)}}\\
  &\quad + \sum_{l=L_1}^{L_2}l^{\frac{2-\beta_j-\alpha_j-\min\{\gamma_j, \alpha_1\}}{\beta_j-1}+\eps}\,.
\end{align*}
By assumption $\alpha_j>\beta_j/2$, so $\tfrac{2-\beta_j-2\alpha_j}{2(\beta_j-1)}<-1$ and therefore the first sum is $o_{L_1\to \infty}(1)$. Moreover, since $\alpha_j\geq \alpha_1>1/2$ and $\gamma_j>1-\alpha_j$ we can choose $\eps>0$ so that for each $j$
\begin{equation*}
  \frac{2-\beta_j-\alpha_j-\min\{\gamma_j, \alpha_1\}}{\beta_j-1}+\eps<\frac{2-\beta_j - \alpha_j - (1-\alpha_j)}{\beta_j-1} = -1\,.
\end{equation*}
Thus also the second sum is $o_{L_1\to \infty}(1)$. Since $L_1 \to \infty$ as $N_1\to \infty$ we conclude that $\{\log(R(n))\}_{n\geq 1}$ is Cauchy, which completes the proof of Theorem~\ref{thm: pure ac}.\qed

\appendix

\section{On Kr\"uger's Proposition 3.2}
\label{app: correction to Kruger}

Let $V_{\omega, \alpha, \beta}(n) = \frac{\cos(2\pi \omega n^\beta)}{n^\alpha}$. Assume that $0<\alpha<\frac{1}{2}- \frac{\beta-1}{2}$ and $1<\beta <2$. Introduce the intervals
\begin{align*}
  \Lambda_k^0 &= [2^k, 2^{k+1}]\\
  \Lambda_k^- &= [2^{k-1}, 2^{k}-1]\\
  \Lambda_k^+ &= [2^{k+1}+1, 2^{k+1}+2^{k-1}]
\end{align*}
and set $\Lambda_k = \Lambda_k^- \cup \Lambda_k^0 \cup \Lambda_l^+$.

\begin{proposition}
  Fix $\alpha, \beta$ as above and $\omega, \lambda\in \R$, and $\omega' \in [0, 1)$.
  Set $\eps = \frac{2-\beta-2\alpha}{6}>0$.
  Then for $k$ large enough there exists intervals $I^\pm \subseteq \Lambda_k^\pm$ satisfying
  \begin{enumerate}
    \item $|\#(I^\pm) - 2^{(\alpha + \eps)k+1}|\lesssim 1$, and
    \item $\|\lambda V_{\omega, \alpha, \beta} - 2\lambda_\pm' \cos(2\pi(\phi_\pm + \omega' \cdot))\|_{L^\infty(I^\pm)} \lesssim 2^{-(\alpha +4\eps)k}$ for some $\phi_\pm \in \R$ and with $\lambda_\pm' \in [ \frac{\lambda}{2^{(k-2)\alpha}} , \frac{\lambda}{2^{(k+1)\alpha}}].$
  \end{enumerate}
\end{proposition}
\begin{remark}
  Note that the statement is weaker than that claimed by Kr\"uger in two regards; the length of the interval where the approximation holds is smaller, and the $L^\infty$-bound on that interval is weaker. However, when it comes to the application of the lemma in the proof of Kr\"uger's main result it is good enough. Indeed, the important part is that the interval where the $L^\infty$ error is much smaller than the size of the gaps in the spectrum of the respective almost Mathieu operators around $\pm 2\cos(\pi \omega')$ has a length which is much larger than the reciprocal of the gap. Since away from $E\in \{-2, 0, 2\}$ the gaps in the spectrum of the almost Mathieu operators are $\gtrsim \lambda_\pm'$ (if $\lambda_\pm'$ is small enough) this is the case with the bound in the proposition.
\end{remark}

\begin{proof}
  (The deviation from Kr\"uger's proof is only in details he omits).
  
  The arguments for $I^\pm$ are almost identical so we consider only the case of $I^-$. Define $c = \lfloor \frac{2^{k-1}+2^k-1}{2}\rfloor$, the centre of $\Lambda^-$, and set $\omega_m = \omega\beta m^{\beta-1}$. While Kr\"uger's statement that
  \begin{equation*}
    \frac{d}{dx} x^{\beta-1} = O(x^{\beta-2})
  \end{equation*}
  is trivial, his claim that $\omega_{c+2^{\epsilon k/4}}-\omega_c \to \infty$ is wrong in the considered range of $\alpha, \beta$. Indeed, the difference $\omega_{c+y}-\omega_c$ is bounded as long as $y = O(2^{(2-\beta)k})$. 

  For $\delta>0$ to be determined we have that $\omega_{c+2^{(2-\beta+\delta)k}}-\omega_c \to \infty$. Thus we conclude that there for $k$ large enough exists $\hat m$ with
  \begin{equation*}
     |\hat m -c | \leq 2^{(2-\beta+\delta)k} \quad \mbox{and}\quad (\omega_{\hat m} - \omega') \mod 1\lesssim 2^{(k-1)(\beta-2)}\,.
  \end{equation*}  

  Set $\ell = \lceil 2^{(\alpha+ \eps)k}\rceil$ for $\eps$ as in the proposition and define $I^- = [\hat m-\ell, \hat m + \ell]$. If $\max\{2-\beta+\delta, \alpha + \eps\} <1$ then $I^- \subset \Lambda^-$. In particular, we need to choose $\delta <\beta-1$ and check that $\alpha + \eps < 1$. Since $\alpha + 3\eps = 1-\beta/2<1$ the second requirement is ok.

  It remains to prove that $\delta, \phi_-$, and $\lambda_-$ can be chosen such that the two potentials are close in $L^\infty(I^-)$.

  For $n \in I^-$ a Taylor expand around $\hat m$ implies
  \begin{align*}
    \lambda \frac{\cos(2\pi \omega n^\beta)}{n^\alpha} 
    &=
    \lambda \frac{\cos(2\pi( \omega \hat m^{\beta}+ \omega_{\hat m}(n-\hat m)+ O(2^{(k-1)(\beta-2)}\ell^2)))}{\hat m^\alpha + O(2^{(k-1)(\alpha-1)}\ell)} \\
    &=
     \lambda \frac{\cos(2\pi( \omega \hat m^{\beta}-\omega' \hat m +\omega' n)+ O(2^{(k-1)(\beta-2)}\ell^2))}{\hat m^\alpha + O(2^{(k-1)(\alpha-1)}\ell)} \\
    &=
    \lambda \frac{\cos(2\pi( \omega \hat m^{\beta}-\omega' \hat m +\omega' n))}{\hat m^\alpha}  + O\Bigl(\frac{2^{(k-1)(\beta-2)}\ell^2}{\hat m^\alpha} + \frac{2^{(k-1)(\alpha-1)}\ell}{\hat m^{2\alpha}}\Bigr)\,.
  \end{align*}
  Setting $\phi_- = \omega \hat m^\beta -\omega' \hat m$ and $\lambda_-' = \frac{\lambda}{2\hat m^\alpha}$ we conclude that 
  \begin{align*}
    \|\lambda V_{\omega, \alpha, \beta} - 2\lambda'_- \cos(2\pi (\phi_- +\omega' \cdot))\|_{L^\infty(I^-)} 
    &\lesssim
    \frac{2^{(k-1)(\beta-2)}\ell^2}{\hat m^\alpha} + \frac{2^{(k-1)(\alpha-1)}\ell}{\hat m^{2\alpha}}\\
    &=
    2^{(k-1)(\beta-2)+ 2k(\alpha + \eps) - (k-1) \alpha}\\
    &\quad  + 2^{(k-1)(\alpha-1)+k(\alpha+\eps) -2\alpha (k-1)}
  \end{align*}
  For both terms to go to zero faster than $2^{-(\alpha + 4\eps)k}$ we need verify that
  \begin{align*}
    \beta-2+2\alpha + 2\eps-\alpha \leq -\alpha-4\eps \quad \mbox{and} \quad \alpha-1 +\alpha +\eps -2\alpha \leq -\alpha-4\eps\,.
  \end{align*}
  The first inequality is easily checked to hold, in fact it is an equality, and the second is equivalent to
  \begin{align*}
    4-4\alpha-5\beta <0
  \end{align*}
  which is true in the considered parameter range. Since the only assumption concerning $\delta$ is that $0<\delta < \beta-1$ we can choose $\delta = \frac{\beta-1}{2}$, completing the proof.
\end{proof}


\bibliographystyle{amsplain}

\begin{thebibliography}{20}

\bibitem{ChristKiselev_JAMS_1998}
M.~Christ and A.~Kiselev, \emph{Absolutely continuous spectrum for
  one-dimensional {Schr\"odinger} operators with slowly decaying potentials:
  some optimal results}, J. Amer. Math. Soc.
  \textbf{11} (1998), no.~4, 771--797.

\bibitem{DeiftKillip_CMP1999}
P.~Deift and R.~Killip, \emph{On the {Absolutely} {Continuous} {Spectrum} of
  {One}-{Dimensional} {Schrödinger} {Operators} with {Square} {Summable}
  {Potentials}}, Comm. Math. Phys. \textbf{203} (1999),
  no.~2, 341--347.

\bibitem{Delyon_etal_JPhysA_1983}
F.~Delyon, H.~Kunz, and B.~Souillard, \emph{One-dimensional wave equations in
  disordered media}, J. Phys. A \textbf{16} (1983), no.~1, 25--42.

\bibitem{DelyonSimonSouillard_PRL84}
F.~Delyon, B.~Simon, and B.~Souillard, \emph{From {Power}-{Localized} to
  {Extended} {States} in a {Class} of {One}-{Dimensional} {Disordered}
  {Systems}}, Phys. Rev. Lett. \textbf{52} (1984), no.~24, 2187--2189.

\bibitem{delyonSimonSouillard_AHP85}
F.~Delyon, B.~Simon, and B.~Souillard, \emph{From power pure point to continuous spectrum in disordered
  systems}, Ann. Inst. H. Poincar\'e Phys. Th\'eor. \textbf{42} (1985), no.~3, 283--309.

\bibitem{EasthamKalf}
M.~S.~P.~Eastham and H.~Kalf, \emph{Schr\"{o}dinger-type operators with
  continuous spectra}, Research Notes in Mathematics, vol.~65, Pitman (Advanced
  Publishing Program), Boston, Mass.-London, 1982.

\bibitem{FrankLarson_KP_2021}
R.~L. Frank and S.~Larson, \emph{{On the spectrum of the Kronig--Penney model
  in a constant electric field}}, Propbab. Math. Phys. (to appear).

\bibitem{FrankSimon_17}
R.~L.~Frank and B.~Simon, \emph{Eigenvalue bounds for {S}chr\"{o}dinger
  operators with complex potentials. {II}}, J. Spectr. Theory \textbf{7}
  (2017), no.~3, 633--658.

\bibitem{GilbertPearson_subordinacy_1987}
D.~J. Gilbert and D.~B. Pearson, \emph{On subordinacy and analysis of the
  spectrum of one-dimensional {Schr\"odinger} operators}, J.
  Math. Anal. Appl. \textbf{128} (1987), no.~1, 30--56.

\bibitem{Jecko}
T.~Jecko, \emph{On Schr\"odinger and Dirac operators with an oscillating potential}. With contributions by Aiman Mbarek. Rev. Roumaine Math. Pures Appl. \textbf{64} (2019), no.~2-3, 283-314.

\bibitem{Kiselev_CMP_96}
A.~Kiselev, \emph{Absolutely continuous spectrum of one-dimensional
  {S}chr\"{o}dinger operators and {J}acobi matrices with slowly decreasing
  potentials}, Comm. Math. Phys. \textbf{179} (1996), no.~2, 377--400.

\bibitem{kiselev_stability_1998}
A.~Kiselev, \emph{Stability of the absolutely continuous spectrum of the
  {Schr\"odinger} equation under slowly decaying perturbations and a.e.
  convergence of integral operators}, Duke Math. J. \textbf{94}
  (1998), no.~3, 619--646.

\bibitem{KiselevLastSimon_EFGP_1998}
A.~Kiselev, Y.~Last, and B.~Simon, \emph{Modified {Pr\"ufer} and {EFGP}
  {Transforms} and the {Spectral} {Analysis} of {One}-{Dimensional}
  {Schr\"odinger} {Operators}}, Comm. Math. Phys.
  \textbf{194} (1998), no.~1, 1--45.

\bibitem{KotaniUshiroya_CMP_1988}
S.~Kotani and N.~Ushiroya, \emph{One-{Dimensional} {Schrödinger} {Operators}
  with {Random} {Decaying} {Potentials}}, Comm. Math. Phys. \textbf{115} (1988), 247--266.

\bibitem{Kruger_2011}
H.~Kr\"uger, \emph{Schrödinger operators with potential
  $v(n)=n^{-\gamma}\cos(2\pi n^\rho)$}, Contemporary {Mathematics} (R.~Sims and
  D.~Ueltschi, eds.), vol. 552, Amer. Math. Soc., Providence,
  RI, 2011, pp.~109--116.

\bibitem{Lukic_CMP_2011}
M.~Lukic, \emph{Orthogonal {Polynomials} with {Recursion} {Coefficients} of
  {Generalized} {Bounded} {Variation}}, Comm. Math. Phys.
  \textbf{306} (2011), no.~2, 485--509.

\bibitem{Naboko_86}
S.~N. Naboko, \emph{Dense {Point} {Spectra} of {Schr\"odinger} and {Dirac}
  {Operators}}, Teoret. Mat. Fiz. \textbf{68} (1986), no.~1, 18--28.

\bibitem{ReedSimon_IV}
M.~Reed and B.~Simon, \emph{Methods of modern mathematical physics. {IV}.
  {A}nalysis of operators}, Academic Press, New York-London, 1978.

\bibitem{Remling_CMP1998}
C.~Remling, \emph{The {Absolutely} {Continuous} {Spectrum} of
  {One}-{Dimensional} {Schr\"odinger} {Operators} with {Decaying}
  {Potentials}}, Comm. Math. Phys. \textbf{193} (1998),
  no.~1, 151--170.

\bibitem{Simon_CMP_83}
B.~Simon, \emph{Some {J}acobi matrices with decaying potential and dense point
  spectrum}, Comm. Math. Phys. \textbf{87} (1983), no.~2, 253--258.

\bibitem{Simon_PAMS97}
B.~Simon, \emph{Some {Schr\"odinger} operators with dense point spectrum}, Proc. Amer. Math. Soc.
  \textbf{125} (1997), no.~1, 203--208 (en).

\bibitem{Stolz_SlowlyOsc_1994}
G.~Stolz, \emph{Spectral theory for slowly oscillating potentials {I}. {Jacobi}
  matrices}, Manuscripta Math. \textbf{84} (1994), no.~1, 245--260.

\bibitem{Titchmarsh_RiemanZeta86}
E.~C. Titchmarsh, \emph{The theory of the {R}iemann zeta-function}, second ed.,
  The Clarendon Press, Oxford University Press, New York, 1986, Edited and with
  a preface by D. R. Heath-Brown.

\bibitem{vNeumannWigner_1929}
J.~von Neumann and E.~P.~Wigner, \emph{\"Uber merkw\"urdige diskrete Eigenwerte}, Phys. Z. \textbf{30}
(1929), 465-–467.

\bibitem{White_TAMS1983}
D.~A.~W. White, \emph{Schr\"odinger {Operators} with {Rapidly} {Oscillating}
  {Central} {Potentials}}, Trans. Amer. Math. Soc.
  \textbf{275} (1983), no.~2, 641--677.

\end{thebibliography}
\def\myarXiv#1#2{\href{http://arxiv.org/abs/#1}{\texttt{arXiv:#1\,[#2]}}}

\end{document}